\documentclass[psamsfonts, 11pt]{amsart}

\usepackage{amsmath, amsthm, amsfonts, amssymb, graphicx, url}

\newtheorem{theorem}{Theorem}
\newtheorem{lemma}[theorem]{Lemma}

\newtheorem{proposition}[theorem]{Proposition}

\theoremstyle{definition}

\newcommand{\co}{\colon\thinspace}
\DeclareMathOperator{\Area}{Area}
\DeclareMathOperator{\const}{const}
\DeclareMathOperator{\Vol}{Vol}

\begin{document}

\title[Macroscopic stability and simplicial norms of hypersurfaces]{Macroscopic stability and simplicial norms of hypersurfaces}
\author{Hannah Alpert}
\address{231 West 18th Avenue, Columbus, OH 43210}
\email{alpert.19@osu.edu}

\begin{abstract}
We introduce a $\mathbb{Z}$--coefficient version of Guth's macroscopic stability inequality for almost-minimizing hypersurfaces.  In manifolds with a lower bound on macroscopic scalar curvature, we use the inequality to prove a lower bound on areas of hypersurfaces in terms of the Gromov simplicial norm of their homology classes.  We give examples to show that a very positive lower bound on macroscopic scalar curvature does not necessarily imply an upper bound on the areas of minimizing hypersurfaces.
\end{abstract}

\maketitle

\section{Introduction}

The main theorem of this paper is the following.

\begin{theorem}\label{main-thm}
Let $(M, g)$ be a closed Riemannian manifold of dimension $n$, with the property that in the universal cover of $M$ every ball of radius $1$ has volume less than or equal to some number $V_0$.  We also assume that $\pi_1(M)$ is residually finite.  Let $\Sigma$ be a $2$--sided embedded closed hypersurface in $M$, and let $i \co \Sigma \hookrightarrow M$ denote the inclusion.  Then we have the inequality
\[\Area \Sigma \geq \const(n, V_0) \cdot \Vert i_*[\Sigma]\Vert_{\Delta},\]
where the constant is positive and $\Vert i_*[\Sigma]\Vert_{\Delta}$ denotes the Gromov simplicial norm of the homology class $i_*[\Sigma] \in H_{n-1}(M; \mathbb{R})$ corresponding to $\Sigma$.
\end{theorem}

This theorem is a macroscopic analogue of theorems about minimal surfaces.  The hypothesis about volumes of unit balls in the universal cover can be thought of as a lower bound on macroscopic scalar curvature; just as a lower bound on scalar curvature gives an upper bound on volumes of infinitesimal balls, here we have an upper bound on volumes of larger-scale balls.  Gromov has suggested that theorems about scalar curvature should have analogues involving volumes of unit balls; see Guth's exposition~\cite{Guth10'} which includes a definition of the term ``macroscopic scalar curvature''.

The $3$--dimensional version of Theorem~\ref{main-thm} with minimal surfaces says that if the scalar curvature of $M^3$ is at least that of $\mathbb{H}^2 \times \mathbb{R}$, then every $2$--sided stable minimal hypersurface has area at least that of the hyperbolic surface of the same genus; if the scalar curvature of $M$ is at least that of $S^2 \times \mathbb{R}$, then every $2$--sided stable minimal hypersurface is a sphere and has area at \emph{most} that of $S^2$~\cite{Schoen79, Shen97, Nunes13}.  This second statement, about positive scalar curvature, is false for macroscopic scalar curvature, as we discuss in Propositions~\ref{top} and~\ref{area} at the end of the paper.

In higher dimensions, comparing the scalar curvature to that of $\mathbb{H}^{n-1} \times \mathbb{R}$ gives a comparison between the $(n-1)$--dimensional area of a $2$--sided stable minimal hypersurface and its Yamabe invariant~\cite{Cai01, Moraru16}.  In Theorem~\ref{main-thm}, we compare the area of a hypersurface to its simplicial norm.  Because this  simplicial norm depends on the homology class of the embedding rather than on the intrinsic topology of the hypersurface, there is no need to include a hypothesis about the minimality of $\Sigma$.  

The Gromov simplicial norm, introduced in~\cite{Gromov82}, is a semi-norm on singular homology with coefficients in $\mathbb{R}$.  Given a homology class $h$, the \textbf{\textit{simplicial norm}} $\Vert h \Vert_{\Delta}$ is defined to be the infimum over all cycles $c$ representing $h$ of the sum of absolute values of coefficients in $c$.  It is often zero.  For closed manifolds the simplicial norm of the fundamental homology class is called the simplicial volume, and for closed hyperbolic manifolds the simplicial volume is a constant times the hyperbolic volume; the constant is the volume of a regular ideal simplex in $n$--dimensional hyperbolic space~\cite{Gromov82}.  Thus, like the Yamabe invariant, the simplicial norm can be thought of as a topological generalization of hyperbolic volume.  It has the nice property of being multiplicative under covering maps.  For arbitrary continuous maps $f \co X \rightarrow Y$ we have $\Vert f_*h \Vert_{\Delta} \leq \Vert h \Vert_{\Delta}$ for all $h \in H_*(X; \mathbb{R})$.

Taking Theorem~\ref{main-thm} in the case where $M$ is diffeomorphic to the product of a hyperbolic $(n-1)$--manifold with $S^1$, and the hypersurface $\Sigma$ is homologous to the hyperbolic factor, we get a comparison between the area of $\Sigma$ and the hyperbolic area of the hyperbolic factor.  Thus Theorem~\ref{main-thm} is a generalization of the main theorem from the paper~\cite{Alpert17} of the present author with Kei Funano, under the additional hypothesis that $\Sigma$ is $2$--sided; for $1$--sided $\Sigma$, the methods of that paper still prove something more.

In Section~\ref{section-main} we prove Lemma~\ref{main-lemma}, the macroscopic stability inequality with $\mathbb{Z}$--coefficients, based on the stability lemma from Guth's paper~\cite{Guth10}.  We use it to prove Theorem~\ref{main-thm} using a theorem from Guth's paper~\cite{Guth11}.  Then in Section~\ref{section-positive} we discuss the case of positive lower bounds on macroscopic scalar curvature.  An upper bound on the areas of minimizing hypersurfaces seems to be false, but in Proposition~\ref{surface-prop} we give an upper bound on the areas of a $2$--dimensional surface in terms of areas of balls in the universal covers of balls in the surface.

\emph{Acknowledgments.}  The questions explored in this paper were suggested by Andr\'e Neves, who also pointed out the references about areas of minimal surfaces under scalar-curvature hypotheses.

\section{Stability lemma and bounded negative curvature}\label{section-main}

We prove Theorem~\ref{main-thm} using a macroscopic stability inequality.  For stable minimal surfaces, the curvature of the ambient space and the curvature of the surface are related by the second variation formula and the Gauss equations.  Here we measure curvature macroscopically in terms of volumes of balls, and instead of stable minimal hypersurfaces we use hypersurfaces that are almost area-minimizing in their homology class.  The macroscopic stability inequality we use here is an adaptation of the one introduced by Guth in~\cite{Guth10}.  There he uses homology with coefficients in $\mathbb{Z}_2$, and here we use coefficients in $\mathbb{Z}$.  As in~\cite{Guth10} we define the \textbf{\textit{length}} $L(\alpha)$ of any cohomology class $\alpha$ in degree $1$ to be the infimal length of $1$--cycles $c$ with $\alpha(c) \neq 0$, and say that a hypersurface $\Sigma$ is \textbf{\textit{minimizing up to $\delta$}} if it is within $\delta$ of the infimal area of hypersurfaces in its homology class.

\begin{lemma}[Macroscopic stability inequality with $\mathbb{Z}$--coefficients]\label{main-lemma}
Let $(M, g)$ be a closed Riemannian manifold of dimension $n$.  Let $\alpha \in H^1(M; \mathbb{Z})$.  Suppose that $Z \subset M$ is a smooth embedded oriented hypersurface, Poincar\'e dual to $\alpha$ and minimizing up to $\delta$ (among embedded surfaces in its homology class).  Suppose that $R < \frac{1}{2}L(\alpha)$.  Then for any $p \in \Sigma$ we have
\[\Area B_{(\Sigma, g \vert_{\Sigma})}(p, R/2) \leq 2 R^{-1} \Vol B_{(M, g)}(p, R) + \delta.\]
\end{lemma}

We prove this lemma with $\mathbb{Z}$--coefficients because the Gromov simplicial norm does not work well with $\mathbb{Z}_2$--coefficients.  We can count the nonzero coefficients in $\mathbb{Z}_2$--cycles, but the behavior under covering maps is not good.  Therefore, in order to use the simplicial norm in Theorem~\ref{main-thm} it is better to have the lemma with $\mathbb{Z}$--coefficients.  I hope that the version with $\mathbb{Z}$--coefficients will be more useful in other situations as well.

In Guth's version, the left-hand side of the inequality measures the area, not of a ball taken inside $\Sigma$, but of the portion of $\Sigma$ inside a ball in $M$.  That version of the inequality may be false for $\mathbb{Z}$--coefficients.  For instance, in a ball in $M$, the surface $\Sigma$ may have several parallel sheets.  Under $\mathbb{Z}_2$--coefficients, the sheets cancel and are homologous to something smaller, but under $\mathbb{Z}$--coefficients they do not cancel, if they are oriented in the same direction.  Thus, in the version with $\mathbb{Z}$--coefficients we show that each of these sheets inside a ball in $M$ is not too large, which implies that the balls taken inside $\Sigma$ are not too large.

\begin{proof}[Proof of Lemma~\ref{main-lemma}]
As in~\cite{Guth10}, we use the coarea inequality to select $t \in (\frac{1}{2}R, R)$ such that the sphere $S(p, t)$ in $M$ has area at most $2R^{-1}\Vol B_{(M, g)}(p, R)$, and such that $S(p, t)$ is smooth and has transverse intersection with $\Sigma$.

As in~\cite{Guth10}, the assumption $R < \frac{1}{2}L(\alpha)$ implies that every loop in the ball $B(p, t)$ has a zero signed intersection number with $\Sigma$.  We group the components of $B(p, t) \setminus \Sigma$ into ``levels'' $L_1, \ldots, L_k$ and group the components of $\Sigma \cap B(p, t)$ into ``dividers'' $D_1, \ldots, D_{k-1}$, such that for any path from level $L_i$ to level $L_j$, its signed intersection number with $\Sigma$ is $j-i$, and each divider $D_i$ is the boundary between levels $L_{i}$ and $L_{i+1}$.  We think of $D_0$ and $D_k$ as empty dividers before level $L_1$ and after level $L_k$.

For each divider $D_i$, we consider all the ways to modify $D_i$ by adding a $\mathbb{Z}$--linear combination of boundaries of levels, and let $D'_i$ be the way with the least area.  Then $D'_i$ consists of some $D_{j_i}$ (with $0 \leq j_i \leq k$) plus the portion of $S(p, t)$ lying between dividers $D_i$ and $D_{j_i}$.  Notice that the sequence $j_1, \ldots, j_{k-1}$ is (weakly) monotonic---the new surfaces $D'_1, \ldots, D'_{k-1}$ may cover some dividers and pieces of $S(p, t)$ multiple times but they do not cross.  Thus we can modify them very slightly so that replacing $D_1, \ldots, D_{k-1}$ by $D'_1, \ldots, D'_{k-1}$ in $\Sigma$ yields a new smooth embedded hypersurface homologous to $\Sigma$.

Because $\Sigma$ is area-minimizing up to $\delta$, we have
\[\sum_{i = 1}^{k-1} \Area D'_i \geq \sum_{i = 1}^{k-1} \Area D_i - \delta.\]
Thus, for each divider $D_i$ we have
\[\Area D_i - \delta \leq \Area D'_i \leq \Area S(p, t) \leq 2R^{-1}\Vol B(p, R).\]
The middle inequality is because $D'_i$ was chosen to be of minimal area when compared with taking $j_i = 0$ or $j_i = k$, both of which result in surfaces contained in $S(p, t)$.

Applying these inequalities to the divider $D_i$ containing the center $p$, we have
\[\Area B_{(\Sigma, g\vert_{\Sigma})}(p, R/2) \leq \Area D_i \leq 2R^{-1}\Vol B_{(M, g)}(p, R) + \delta.\]
\end{proof}

The proof of Theorem~\ref{main-thm} is very similar to the proof of the main theorem in~\cite{Alpert17}; there, we prove the following modified version of a theorem of Guth from~\cite{Guth11}.

\begin{theorem}[\cite{Alpert17, Guth11}]\label{thm-attr}
Let $(\Sigma, g)$ be a closed oriented Riemannian manifold of dimension $n$, such that every ball of radius $1$ in $\Sigma$ has volume at most $V_0$.  Let $A$ be an aspherical topological space (i.e., its universal cover is contractible), and suppose that $f \co \Sigma \rightarrow A$ is a continuous map such that $f(\gamma)$ is null-homotopic for every loop $\gamma \subset \Sigma$ of length at most $1$.  Then we have
\[\Vert f_*[\Sigma]\Vert_{\Delta} \leq \const(n, V_0) \cdot \Vol \Sigma.\]
\end{theorem}

The proof in~\cite{Guth11} is quite technical and involved, but very roughly, the idea is to replace $\Sigma$ by the nerve of a covering by balls such that the number of balls is at most proportional to the volume of $\Sigma$.  Using this theorem and Lemma~\ref{main-lemma} we can finish the proof of Theorem~\ref{main-thm}.

\begin{proof}[Proof of Theorem~\ref{main-thm}]
We find a finite cover $\widehat{M}$ of $M$ such that every loop in $\widehat{M}$ of length at most $2$ is null-homotopic.  This is possible because $\pi_1(M)$ is residually finite: only finitely many elements of $\pi_1(M)$ come from short loops, so we can choose $\pi_1(\widehat{M})$ to be a finite-index subgroup of $\pi_1(M)$ that excludes them all.  Let $\widehat{Z}$ be one connected component of the preimage of $\Sigma$ in $\widehat{M}$.  We claim that it suffices to prove Theorem~\ref{main-thm} for $\widehat{M}$ and $\widehat{\Sigma}$ instead of for $M$ and $\Sigma$.  Indeed, let $\pi \co \widehat{M} \rightarrow M$ denote the covering map, let $\widehat{i} \co \widehat{\Sigma} \hookrightarrow \widehat{M}$ denote the inclusion, and suppose that $\widehat{\Sigma}$ has $k$ sheets over $\Sigma$.  Because the orientation of $\Sigma$ pulls back to $\widehat{\Sigma}$, all $k$ sheets are oriented the same way.  Thus, we know 
\[k \cdot \Area Z = \Area \widehat{Z}\]
and
\[\Vert \widehat{i}_*[\widehat{Z}]\Vert_{\Delta} \geq \Vert \pi_* \widehat{i}_* [\widehat{Z}] \Vert_{\Delta} = \Vert k \cdot i_*[Z] \Vert_{\Delta} = k \cdot \Vert i_*[Z]\Vert_{\Delta},\]
so if we can show
\[\Area \widehat{Z} \geq \const(n, V_0) \cdot \Vert \widehat{i}_*[\widehat{Z}]\Vert_{\Delta},\]
 then we can string together the inequalities to get
 \[k \cdot \Area Z \geq \const(n, V_0) \cdot k \cdot \Vert i_*[Z] \Vert_{\Delta},\]
 which is $k$ times the desired inequality.
 
 Thus we may replace $M$ by $\widehat{M}$ and assume that every loop in $M$ of length at most $2$ is null-homotopic.  Then every unit ball in $M$ has volume at most $V_0$.  We may also assume that $\Sigma$ is minimizing up to $\delta = 1$ among embedded hypersurfaces in its homology class.  Applying Lemma~\ref{main-lemma} with $R = 1$ we find that every ball of radius $\frac{1}{2}$ in $\Sigma$ has volume at most $2V_0 + 1$.
 
 To apply Theorem~\ref{thm-attr} we need a map from $\Sigma$ to an aspherical topological space.  Let $A = K(\pi_1(M), 1)$ and let $f \co \Sigma \rightarrow A$ be the composition of the inclusion $i \co \Sigma \hookrightarrow M$ with the classifying map $M \rightarrow K(\pi_1(M), 1)$.  This classifying map preserves simplicial norm (Mapping Theorem from Section~3.1 of~\cite{Gromov82}, or see~\cite{Ivanov87}), so we have
 \[\Vert f_*[\Sigma] \Vert_{\Delta} = \Vert i_*[\Sigma] \Vert_{\Delta}.\]  
 Loops of length at most $2$ are null-homotopic in $M$ and hence also in $\Sigma$, so we may apply Theorem~\ref{thm-attr} (with $\Sigma$ scaled by $2$ to get a bound on unit balls rather than on balls of radius $\frac{1}{2}$) to get
 \[\Vert f_*[\Sigma] \Vert_\Delta \leq \const(n, 2^{n-1}(2V_0 + 1)) \cdot \Area \Sigma,\]
 and we obtain
 \[\Vert i_*[\Sigma] \Vert_{\Delta} \leq \const(n, V_0) \cdot \Area \Sigma.\]
\end{proof}

\section{Positive curvature}\label{section-positive}

A \emph{positive} lower bound on scalar curvature should give an \emph{upper} bound on the areas of stable minimal hypersurfaces.  In $3$ dimensions, if the scalar curvature is at least that of $S^2 \times \mathbb{R}$, then every $2$--sided stable minimal hypersurface is a sphere and has area at most that of $S^2$.  We can ask whether the same is true for macroscopic scalar curvature: if every ball of radius $1$ in the universal cover of a manifold $M$ has very tiny volume, do the area-minimizing hypersurfaces have constrained topology and an upper bound on their areas?  Proposition~\ref{top} shows that the topology is constrained in some sense, but Proposition~\ref{area} shows that we cannot expect an upper bound on the areas.

\begin{proposition}\label{top}
There exists a constant $V_0 > 0$ such that the following holds.  Suppose that $M$ is a $3$--dimensional Riemannian manifold diffeomorphic to the product of a closed surface $\Sigma$ with $S^1$, such that every ball of radius $1$ in the universal cover of $M$ has volume at most $V_0$.  Then $\Sigma$ must be the sphere or the projective plane.
\end{proposition}

\begin{proof}
Guth proves in~\cite{Guth11} that there is a constant $\delta(n)$ such that for every closed aspherical Riemannian $n$--manifold, the volumes of unit balls in the universal cover are at least $\delta(n)$.  If $\Sigma$ is a surface other than the sphere or the projective plane, then $\Sigma \times S^1$ is aspherical, so we may take $V_0 = \delta(3)$ to get the desired result.  Alternatively, Guth proves in~\cite{Guth10} a similar estimate for $n$--manifolds where there is a nonzero $n$--fold cup product.
\end{proof}

Although having a very positive macroscopic scalar curvature constrains the topology of a manifold and its hypersurfaces, it does not necessarily result in an upper bound on the areas of stable minimal hypersurfaces.

\begin{proposition}\label{area}
For any $V_0 > 0$, any $K > 0$, and any $n$, there exists an $n$--dimensional closed Riemannian manifold $M$ such that every ball of radius $1$ in the universal cover of $M$ has volume at most $V_0$, but simultaneously $M$ contains an area-minimizing hypersurface of area at least $K$.
\end{proposition}

\begin{proof}
Let $\Sigma$ be a Riemannian $(n-1)$--manifold diffeomorphic to $S^{n-1}$, obtained by taking the boundary of a very small neighborhood of a very long segment in $\mathbb{R}^n$.  Then let $M = \Sigma \times S^1$.  By making the neighborhood very small we can guarantee that the unit balls in the universal cover of $M$ have volume at most $V_0$, and by making the segment very long we can guarantee that the area of $\Sigma$ is at least $K$.
\end{proof}

For $n = 3$, the fact that these minimizing hypersurfaces have unbounded area implies that the ambient manifolds do not have a positive lower bound on scalar curvature.  Indeed, along the length of the skinny sphere, the sphere looks like a cylinder, so the manifold looks locally flat.  And yet, the unit balls in the universal cover do not look Euclidean.  They look just like the unit balls downstairs and have very small volume.  The reason why they do not look flat is that the universal cover is not a local construction.  If we were to take the universal cover of only the cylindrical portion, then the balls would indeed be flat Euclidean balls.

This example suggests that maybe instead of taking the volumes of balls in the universal cover of the manifold $M$, we should take the volumes of balls in the universal covers of \emph{balls}.  The resulting number, like scalar curvature, would depend only on local properties of $M$.  However, it turns out that we can construct an example like the one in Proposition~\ref{area} with the additional property that the unit balls in the ambient manifold are simply connected.  Thus, there does not seem to be an upper bound on volumes of minimizing hypersurfaces in terms of macroscopic scalar curvature information.

The modified example is constructed as follows.  We start with a skinny sphere $\Sigma$ as in Proposition~\ref{area}, constructed as the boundary of the $\varepsilon$--neighborhood of a segment of length $L$, where $\varepsilon$ is very small and $L$ is very large.  Let $\Sigma'$ be the disjoint union of $10L$ copies of the boundary of the $\varepsilon$--neighborhood of a segment of length $\frac{1}{10}$.  Then we take $M$ to be the disjoint union of $\Sigma \times [-\frac{1}{10}, \frac{1}{10}]$ and $\Sigma' \times [\frac{3}{10}, 2 - \frac{3}{10}]$, connected by two short identical pieces that interpolate between them.  Each of these connector pieces looks like a Morse function from $\frac{1}{10}$ to $\frac{3}{10}$, with $10L - 1$ critical points at $\frac{2}{10}$, as the level set $\Sigma$ pinches off to become the level set $\Sigma'$, like sausage links.  Together, $M$ looks like a bundle over the circle $\mathbb{R}/2\mathbb{Z}$ of length $2$, with the fibers over the interval $(-\frac{2}{10}, \frac{2}{10})$ diffeomorphic to $\Sigma$, the fibers over the interval $(\frac{2}{10}, 2 - \frac{2}{10})$ diffeomorphic to $\Sigma'$, and singular fibers at $\pm \frac{2}{10}$.  We can adjust the metric to ensure that $\Sigma$ is area-minimizing in its homology class, despite being of large area, and also ensure that every unit ball in $M$ with center in $\Sigma$ is simply connected with tiny volume.

In two dimensions, though, a condition about balls in the universal covers of balls does guarantee that the surface has small total area, as follows.

\begin{proposition}\label{surface-prop}
For every $\varepsilon$ with $0 < \varepsilon < \frac{1}{8}$, the following holds.  Let $M$ be a closed surface with the property that for any point $p \in M$, the ball of radius $1$ around any lift $\widetilde{p}$ of $p$ in the universal cover $\widetilde{B(p, 1)}$ of the unit ball at $p$ has area at most $\varepsilon$.  Then the area of $M$ is at most $8\varepsilon$.
\end{proposition}

\begin{proof}
First we claim that the diameter of $M$ is at most $2$.  Suppose to the contrary that there are two points in $M$ with distance greater than $2$, and let $p$ be the midpoint of a length-minimizing geodesic between them.  Using the coarea inequality we may select $t \in (\frac{1}{2}, 1)$ such that the total length of the sphere $S(\widetilde{p}, t)$ in $\widetilde{B(p, 1)}$ is at most $2\varepsilon$.

The sphere $S(\widetilde{p}, t)$ is a disjoint union of loops, and because $\widetilde{B(p, 1)}$ is simply connected, each loop bounds a disk.  Exactly one of the loops in $S(\widetilde{p}, t)$ encloses $\widetilde{p}$, and that loop has length at most $2\varepsilon$.  The geodesic through $p$ lifts to a geodesic through $\widetilde{p}$, and the portion inside the loop has length at least $2t$.  Thus, because $\varepsilon < \frac{1}{2}$, we have $2\varepsilon < 2t$ and we can homotope the geodesic to a shorter path by following the loop, giving a contradiction.

Thus, the diameter of $M$ is at most $2$.  Next we describe how to cover $M$ by at most eight unit balls.  Then the area of $M$ is at most $8\varepsilon$, proving the proposition.

Start with $p_1 \in M$ arbitrary.  Find $t_1 \in (\frac{1}{2}, 1)$ such that the sphere $S(\widetilde{p_1}, t_1)$ in $\widetilde{B(p_1, 1)}$ has total length at most $2\varepsilon$.  Find the loop $L_1$ in $S(\widetilde{p_1}, t_1)$ that encloses $\widetilde{p_1}$, and select $p_2 \in M$ to be any point on the image of $L_1$ in $M$.  

Next we consider the universal cover of $B(p_1, 1) \cup B(p_2, 1)$.  Each component of the preimage of $B(p_2, 1)$ in this covering space is covered by $\widetilde{B(p_2, 1)}$, so the ball of radius $1$ around $\widetilde{p_2}$ in the universal cover of $B(p_1, 1) \cup B(p_2, 1)$ has area at most $\varepsilon$.  Thus we may find $t_2 \in (\frac{1}{2}, 1)$ such that the sphere $S(\widetilde{p_2}, t_2)$ has length at most $2\varepsilon$, and select $p_3 \in M$ to be on the image in $M$ of the loop $L_2$ in $S(\widetilde{p_2}, t_2)$ enclosing $\widetilde{p_2}$.

Continue in this way: having selected $p_k$, take the universal cover of the union $B(p_1, 1) \cup \cdots \cup B(p_k, 1)$, select $t_k \in (\frac{1}{2}, 1)$ such that $S(\widetilde{p_k}, t_k)$ has length at most $2\varepsilon$, and select $p_{k+1} \in M$ to be on the image of the loop $L_k$ in $S(\widetilde{p_k}, t_k)$ that encloses $\widetilde{p_k}$.  Once the union $B(p_1, 1) \cup \cdots \cup B(p_k, 1)$ is all of $M$, stop.

We claim that the union $B(p_1, 1) \cup \cdots \cup B(p_k, 1)$ covers the ball of radius $\sum_{i = 1}^k (t_i - 2\varepsilon)$ around $p_1$.  Notice that $L_1$ encloses at least at $t_1$--neighborhood of $\widetilde{p_1}$.  Then, $L_2$ encloses at least a $(t_2 - 2\varepsilon)$--neighborhood around (the lift of the projection of) $L_1$, and in general each $L_i$ encloses at least at $(t_i - 2\varepsilon)$--neighborhood around $L_{i-1}$.  Thus, let $q$ be any point outside $B(p_1, 1) \cup \cdots \cup B(p_k, 1)$, and find a length-minimizing geodesic $\gamma$ from $p_1$ to $q$.  Lifting $\gamma$ to the universal cover of $B(p_1, 1)$, we can cut off a segment extending from $\widetilde{p_1}$ to $L_1$, with length at least $t_1$.  Then, lifting $\gamma$ to the universal cover of $B(p_1, 1) \cup B(p_2, 1)$, we can cut off a segment extending from $L_1$ to $L_2$, of length at least $t_2 - 2\varepsilon$, and so on.  Reaching $L_k$ we have cut off length at least $t_1 + (t_2 - 2\varepsilon) + \cdots + (t_k - 2\varepsilon)$ and still have not exhausted $\gamma$, so the distance from $p_1$ to $q$ must be greater than $\sum_{i = 1}^k (t_i - 2\varepsilon)$.

We know that the diameter of $M$ is at most $2$, and we have assumed that $\varepsilon < \frac{1}{8}$.  Thus each $t_i - 2\varepsilon$ is at least $\frac{1}{2} - \frac{1}{4} = \frac{1}{4}$, so $B(p_1, 1) \cup \cdots \cup B(p_8, 1)$ must cover all of $M$.
\end{proof}

\bibliography{z-stability-bib}{}
\bibliographystyle{amsalpha}
\end{document}